\documentclass[12pt]{article}
\usepackage{amsmath, amsfonts, amsthm, amssymb, color, hyperref, extarrows, enumitem, nicefrac}

\newtheorem{theorem}{Theorem}[section]
\newtheorem{definition}[theorem]{Definition}
\newtheorem{cor}[theorem]{Corollary}
\newtheorem{lem}[theorem]{Lemma}

\numberwithin{equation}{section}

\newtheorem{example}{Example}

\usepackage[hyperpageref]{backref}

\usepackage{cite}

\hypersetup{
	colorlinks   = true,
	citecolor    = magenta}
	
\begin{document}
\title{\vspace{-1cm} \bf Unique continuation for $\bar\partial$ with square-integrable potentials \rm}
\author{Yifei Pan\ \  and\ \   Yuan Zhang
}
\date{}

\maketitle

\begin{abstract}
  In this paper, we investigate the unique continuation property for the inequality $|\bar\partial u| \le V|u|$, where $u$ is a vector-valued function from a domain in  $\mathbb C^n$ to $\mathbb C^N$, and  the potential $V\in L^2$. We show that  the strong unique continuation property holds when $n=1$, and the weak unique continuation property holds  when $n\ge 2$. In both cases, the $L^2$ integrability condition on the potential is optimal.   

\end{abstract}

\renewcommand{\thefootnote}{\fnsymbol{footnote}}
\footnotetext{\hspace*{-7mm}
\begin{tabular}{@{}r@{}p{16.5cm}@{}}
& 2010 Mathematics Subject Classification. Primary 32W05; Secondary 35A23, 35A02.\\
& Key words and phrases. Cauchy-Riemann equation, unique continuation, Hardy-Littlewood-Sobolev inequality. 
\end{tabular}}

\section{Introduction}

 Let  $u=(u_1, \ldots, u_N)$ be a vector-valued function from a domain $\Omega\subset \mathbb C^n$ into $\mathbb C^N$. We say a differential equation or inequality 
 satisfies the strong unique continuation property, if every solution that vanishes to infinite order at a point $z_0\in \Omega$ vanishes identically. Here a square-integrable  function $u$ is said to vanish to infinite order (or flat) at  $z_0\in \Omega$  if for all $m\ge 0$,
\begin{equation}\label{flat}
    \lim_{r\rightarrow 0} r^{-m}\int_{|z-z_0|<r}|u(z)|^2 dv_z =0,
\end{equation} 
where $dv_z$  is the Lebesgue  measure element in $\mathbb C^n$ with respect to the dummy variable $z$. 
A differential equation or inequality is said to satisfy the weak unique continuation property, if every solution  that vanishes in an open subset vanishes identically.

While studying the boundary regularity and the 
uniqueness for CR-mappings of hypersurfaces  in \cite{BL90}, Bell and Lempert had proved and applied the strong  unique continuation property for $|\bar\partial u|\le C|u|$ with  $C$ a constant.  In this paper, we consider the unique continuation property for  the following general  type:
\begin{equation}\label{eqn}
    |\bar\partial u|\le V|u| \ \ \text{a.e.,}
\end{equation}
where the  potential $V$ is assumed to be locally square-integrable, i.e., $V\in L^2_{loc}(\Omega)$.

 Our first theorem below  concerns the strong unique continuation property in the case when  $n=1$. Denote by $H_{loc}^k(\Omega)$ the standard (local) Sobolev space of functions whose weak derivatives up to order $k$ are in $L_{loc} ^2(\Omega)$.

\begin{theorem}\label{main}
Let $\Omega$ be a  domain in $\mathbb C$. Suppose $u=(u_1, \ldots, u_N): \Omega\rightarrow \mathbb C^N$ with $u\in H^1_{loc}(\Omega)$  and satisfies $ |\bar\partial u|\le V|u|$
 for some $V\in L_{loc}^2(\Omega)$.  If $u$ vanishes to infinite order at $z_0\in \Omega$, then $u$ vanishes identically.
\end{theorem}

This result is optimal in the sense that the  property fails if $V\notin L^2$, as demonstrated by Example \ref{ex} in Section 3.  Moreover, Mandache in \cite{Ma02} constructed  an example (see Example \ref{ex4}) with an $L^p$ potential, $0<p<2$,  where  even the  weak unique continuation property fails. 

Chanillo and Sawyer \cite{CS90} proved, among other results,  the following strong unique continuation property for the differential inequality  
\begin{equation}\label{eqn2}
    |\Delta u|\leq V|\nabla u|
\end{equation}
 with the potential $V\in L^2$. 

\begin{theorem}\label{cs}\cite{CS90}
Let $\Omega$ be a  domain in $\mathbb C$. Suppose $u=(u_1, \ldots, u_N): \Omega\rightarrow \mathbb R^N$ with $u\in H^2_{loc}(\Omega)$  and satisfies $    |\Delta u|\leq V|\nabla u|$
for some $V\in L_{loc}^2(\Omega)$. If $u$ vanishes to infinite order at $z_0\in \Omega$, then $u$ vanishes identically.
\end{theorem}

We note that the two inequalities \eqref{eqn}-\eqref{eqn2} can easily be converted to each other, at least in the smooth category. However, in terms of the unique continuation property, it turns out Theorem \ref{main} and Theorem \ref{cs} are no longer equivalent. In fact, one can obtain  Theorem \ref{cs} from Theorem \ref{main} by applying Theorem \ref{main} to $\partial u$, as a consequence of the facts that  $\Delta=4\bar\partial\partial$ and $|\partial u(z)|=\frac{1}{2}|\nabla u(x,y)|$ (since $u$ is real-valued in Theorem \ref{cs}). 
To deduce Theorem \ref{main} from Theorem  \ref{cs}, it would boil down to the possible existence of flat solutions to $\bar\partial$ given flat data.  Surprisingly,  there is an obstruction to such existence in the flat category. This phenomenon was discovered by  Liu and the two authors in  \cite{LPZ} which constructed  a smooth function $f$ that is flat at $0\in \mathbb C$, yet $\bar\partial u =f$ has no solutions that are flat at 0.

A well-known example of Wolff shows that  the strong unique continuation for \eqref{eqn2} no longer holds when the real dimension of the source domain is larger than $4$. More precisely,  Wolff constructed  in \cite{Wo94} a smooth real-valued function on $\mathbb R^d, d\ge 5$,  which vanishes to infinite order at the origin and satisfies \eqref{eqn2} with $V\in L^d(\mathbb R^d)$.

As an immediate application of Theorem \ref{main},   we obtain  the following weak unique continuation property for \eqref{eqn} with an $L^2$ potential when $n\ge 2$. Here given $u\in H^1_{loc}(\Omega)$, $\bar\partial u$ is understood as a $(0,1)$ form with distribution components.  

\begin{theorem}\label{main2}
Let $\Omega$ be a  domain in $\mathbb C^n$. Suppose $u=(u_1, \ldots, u_N): \Omega\rightarrow \mathbb C^N$ with $u\in H^1_{loc}(\Omega)$  and satisfies $ |\bar\partial u|\le V|u|$ 
 for some $V\in L_{loc}^2(\Omega)$.  If $u $ vanishes  in an open subset of $ \Omega$,   then $u$ vanishes identically.
\end{theorem}

In contrast to the classical strong unique continuation property results for Laplacian,  the integrability assumption on the potential in Theorem \ref{main2} is independent of the dimension $n $ of the source domain. On the other hand, the  weak unique continuation property fails for \eqref{eqn}  with $L^p$ potentials, $0<p<2$, as shown by a two-dimensional Mandache-type  example   (see Example \ref{ex2}  in Section 4). This shows the integrability assumption in Theorem \ref{main2} on the potential  is  optimal for the weak unique continuation property. 

Theorem \ref{main2} can be applied to obtain the uniqueness of solutions as follows. Denote by $\mathbb C_{(0,1)}^N$  the space of $N$-vectors  each of whose components is a $(0,1)$ form. 

\begin{cor}\label{main3}
 Let $\Omega$ be a  domain in $\mathbb C^n$, $h =(h_1, \ldots, h_N)^T: \Omega\rightarrow \mathbb C_{(0,1)}^N $, and  $A = (A_{jk})_{1\le j, k\le N}: \Omega \rightarrow \mathbb C_{(0,1)}^{N^2}$ with $A\in L^2_{loc}(\Omega)$. Suppose    $f, g:  \Omega\rightarrow \mathbb C^N$ with $f, g\in H^1_{loc}(\Omega)$ and are  solutions to  $ \bar\partial u =  A u +h$.  If $f = g$ in an open subset of $\Omega$,   then $f \equiv g$ on $\Omega$.
\end{cor}


\section{Weighted Hardy-Littlewood-Sobolev inequality}

Recall that given $  f\in L^2(\mathbb C )$,  the Hardy-Littlewood Maximal function of $f$, denoted by $Mf$, is given by  $$ Mf(z): = \sup \frac{1}{|D|}\int_D|f(\zeta)|dv_{\zeta}, \ \  z\in \mathbb C ,  $$
where the supremum is taken over all disks $D$ in $\mathbb C $ containing $z$.  The Riesz fractional integral of $  f$ of order $\alpha\in (0, 2)$ is 
$$I_\alpha f(z): = \int_{\mathbb C } \frac{f(\zeta)}{|\zeta -z|^{2-\alpha}}dv_\zeta, \ \  z\in \mathbb C . $$  

\begin{definition}
Given a weight function $V(> 0)$ on a domain $\Omega\subset \mathbb C $, the weighted $L^2(\Omega)$ space with respect to $V$, denoted by $L^2_V(\Omega)$, is the collection of all functions $f$ on $\Omega$ such that 
$$ \|f\|_{L^2_V(\Omega)}: =\left(\int_{\Omega} |f(z)|^2V(z)dv_z\right)^\frac12 <\infty.  $$
\end{definition}
\medskip

A more general but rather technical variant  of the following theorem can be found in \cite{CW85}. We shall  prove the boundedness of 
$$I_{1}f= \int_{\mathbb C} \frac{f(\zeta)}{|\zeta -\cdot |}dv_\zeta $$  from $L^2_{V^{-1}}(\mathbb C)$ space to $L^2_{V}(\mathbb C)$ space with respect to a weight $V\in L^2(\mathbb C)$ through a much simpler approach.  
\medskip

\begin{theorem}\label{whl}
Let $V\in L^2(\mathbb C)$. Then there exists a universal constant $C_0$ such that for any $f\in L^2_{V^{-1}}(\mathbb C)$,  \begin{equation*}
    \|I_1f\|_{L_V^2(\mathbb C)}\le C_0 \|V\|_{L^2(\mathbb C)} \|f\|_{L_{V^{-1}}^2(\mathbb C)}.
\end{equation*}
\end{theorem}

\begin{proof}
Throughout the proof we use $C$ to represent a universal constant, which may be different at different occurrences. We first show that  for any $g\in L^2(\mathbb C)$,  
\begin{equation}\label{half}
    \|I_{\frac{1}{2}}g(z)\|_{L^2_V(\mathbb C)}\le C \|V\|^{\frac{1}{2}}_{L^2(\mathbb C)}\|g\|_{L^2(\mathbb C)}.
\end{equation} 
Without loss of generality, we assume $g\ge 0$. 

For each $z\in \mathbb C$ with $\delta>0$ to be chosen later, write
$$ I_{\frac{1}{2}}g(z) = \int_{|\zeta-z|<\delta} + \int_{|\zeta-z|>\delta} \frac{g(\zeta)}{|\zeta -z|^{\frac{3}{2}}}dv_\zeta=: I +II. $$
By H\"older inequality, 
$$II \le \left(\int_{|\zeta-z|>\delta}|g(\zeta)|^2dv_\zeta\right)^{\frac{1}{2}}\left(\int_{|\zeta-z|>\delta}\frac{1}{|\zeta -z|^{3}}dv_\zeta\right)^{\frac{1}{2}}\le C\|g\|_{L^2(\mathbb C)} \left(\int_{\delta}^\infty \frac{1}{s^2}ds\right)^{\frac{1}{2}} = \frac{C}{\delta^{\frac{1}{2}}}\|g\|_{L^2(\mathbb C)}. $$
For $I$,
\begin{equation*}
\begin{split}
        I=&\sum_{k=1}^\infty \int_{\frac{\delta}{2^{k}}<|\zeta-z|<\frac{\delta}{2^{k-1}}}\frac{g(\zeta)}{|\zeta -z|^{\frac{3}{2}}}dv_\zeta\\
    \le &\sum_{k=1}^\infty \left(\frac{2^{k}}{\delta}\right)^{\frac{3}{2}} \int_{|\zeta-z|<\frac{\delta}{2^{k-1}}}g(\zeta)dv_\zeta\\
    = &\sum_{k=1}^\infty \left(\frac{2^{k}}{\delta}\right)^{\frac{3}{2}}\left|D_{\frac{\delta}{2^{k-1}}}\right|\left(\frac{1}{\left|D_{\frac{\delta}{2^{k-1}}}\right|} \int_{|\zeta-z|<\frac{\delta}{2^{k-1}}}g(\zeta)dv_\zeta\right)\\
    \le & \sum_{k=1}^\infty 2^{-\frac{k}{2}+2}\pi\delta^\frac{1}{2} Mg(z)\\
    =& C\delta^{\frac{1}{2}}Mg (z).
\end{split}
\end{equation*}
Thus we have
\begin{equation*}\label{ih}
I_{\frac{1}{2}}g(z)\le   C\left(\delta^{\frac{1}{2}}Mg(z)+ \delta^{-\frac{1}{2}}\|g\|_{L^2(\mathbb C)}\right).
    \end{equation*}
After choosing $\delta = V^{-1}(z)$ in the above, we further get
\begin{equation*}
I_{\frac{1}{2}}g(z)\le   C\left(V(z)^{-\frac{1}{2}}Mg(z)+ V(z)^{\frac{1}{2}}\|g\|_{L^2(\mathbb C)}\right).
    \end{equation*}
Square both sides of  the above inequality, multiply them by $V$, and then integrate over $\mathbb C$. One has
\begin{equation}\label{ih2}
    \begin{split}
       \left( \int_{\mathbb C} \left|I_{\frac{1}{2}}g(z)\right|^2V(z)dv_z\right)^{\frac{1}{2}}\le &C \left(\|Mg\|_{L^2(\mathbb C)} + \|V\|_{L^2(\mathbb C)}\|g\|_{L^2(\mathbb C)}\right)\\
        \le & C\left (1+\|V\|_{L^2(\mathbb C)}\right)\|g\|_{L^2(\mathbb C)}.
    \end{split}
\end{equation}
In the last inequality, we have used the boundedness of the Maximal function operator in $L^2(\mathbb C)$ space. 

We further employ the standard rescaling technique to get rid of the constant $1$ in the last line of \eqref{ih2}. Indeed, for any $k>0$, replacing $V$ by $k V$ in \eqref{ih2} and then dividing by $k^\frac{1}{2}$, we have
$$  \left(\int_{\mathbb C} \left|I_{\frac{1}{2}}g(z)\right|^2V(z)dv_z\right)^{\frac{1}{2}}\le C\left(\frac{1}{k^\frac{1}{2}}+k^\frac{1}{2}\|V\|_{L^2(\mathbb C)}\right)\|g\|_{L^2(\mathbb C)}. $$
Choosing $k = \frac{1}{\|V\|_{L^2(\mathbb C)}}$, this then completes the proof of (\ref{half}).

Therefore if $V$ is $L^2(\mathbb C)$, $I_{\frac{1}{2}}$ becomes a bounded operator from $L^2(\mathbb C)$ to $L_V^2(\mathbb C)$. Consequently, let  $I^*_{\frac{1}{2}}$ be the dual of $I_{\frac{1}{2}}$ with respect to the following inner product
$$\langle f, g\rangle: = \int_{\mathbb C} f(z)\overline{ g(z)}dv_z\ \ \ \text{for all}\ \ f, g\in L^2(\mathbb C).$$ Then $I^*_{\frac{1}{2}}$ is bounded from $L_{V^{-1}}^2(\mathbb C)$ to $L^2(\mathbb C)$ satisfying for any $f\in L^2_{V^{-1}}(\mathbb C)$,
\begin{equation*}
     \|I^*_{\frac{1}{2}}f(z)\|_{L^2(\mathbb C)}\le C \|V\|^{\frac{1}{2}}_{L^2(\mathbb C)}\|f\|_{L^2_{V^{-1}}(\mathbb C)}.
\end{equation*}Note that $I_\alpha$ is a convolution operator. So one has $I_{\frac{1}{2}}^* =I_{\frac{1}{2}}$.  Making use of the semi-group property of the fractional integrals, we obtain $I_1 = I_{\frac{1}{2}}\circ I_{\frac{1}{2}}^*$ is a bounded operator from $L_{V^{-1}}^2(\mathbb C)$ into $L_V^2(\mathbb C)$ with the desired estimate.


\end{proof}

\section{Proof of Theorem \ref{main}}

Throughout this section, we restrict on the case when $n=1$.  The key ingredient of Chanillo and Sawyer's proof in \cite{CS90} lies in a pointwise inequality in the spirit of a technical result of  Sawyer in \cite{Sa84}, together with a weighted inequality for the Riesz integral operator $I_1$. 
The proof of  our Theorem \ref{main} essentially follows their idea. Instead of using a Sawyer  inequality, we only need  the following trivial identity \eqref{sa} for the Cauchy kernel. 
\begin{equation}\label{sa}
     \frac{1}{z- \zeta} +\sum_{l=0}^{m-1}\frac{z^{l}}{\zeta^{l+1}} = \frac{z^m}{\zeta^m(z-\zeta)},\ \ \text{for all}\ \ \zeta\ne  z\ \text{nor}\ \  0.
\end{equation}
As usual, denote by $D_R$ the disk in $\mathbb C$ centered at $0$ with radius $R$.
\medskip

\begin{lem}\label{prep}
Let $u\in H^1(\mathbb C)$ has  compact support. Then for almost every $z\in \mathbb C$, 
\begin{equation}\label{cg}
    u(z) = \frac{1}{\pi}\int_{\mathbb C}\frac{\bar\partial u(\zeta)}{z -\zeta }dv_{\zeta}.
\end{equation}
If in addition that $u$ vanishes near $0$, then for any $l\ge 0$, \begin{equation}\label{ll}
    \int_{\mathbb C} \frac{ \bar\partial u(\zeta)}{\zeta^{l+1}}dv_\zeta =0.
\end{equation} \end{lem}

\begin{proof}
Let $R_0>0$ be such that $\text{supp} \ u \subset D_{R_0}$.  For each $z\in \mathbb C$, pick   $R>R_0$ so that $z\in D_R$. If $u\in C^1(\mathbb C)$, by the Cauchy-Green formula (see for instance \cite{Ve}),  
$$u(z) = \frac{1}{2\pi i}\int_{\partial D_R}\frac{u(\zeta)}{\zeta- z}d\zeta -\frac{1}{\pi}\int_{D_R} \frac{\bar\partial u(\zeta)}{\zeta- z}dv_\zeta = \frac{1}{\pi}\int_{\mathbb C} \frac{\bar\partial u(\zeta)}{z-\zeta}dv_\zeta.$$
Here we used the fact that $\text{supp} \ u \subset D_{R_0}$.  For general $u$ in $H_0^1(D_{R_0})$, letting $u_j\in C_c^1(D_{R_0})\rightarrow u$ in $H^1(D_{R_0})$ norm, then for any $R> R_0$, \begin{equation*}
    \begin{split}
        \left \|u - \frac{1}{\pi}\int_{\mathbb C} \frac{\bar\partial u(\zeta)}{\cdot-\zeta}dv_\zeta\right\|_{L^2(D_{R})}\le& \|u-u_j\|_{L^2(D_{R})} + \frac{1}{\pi}\left \| \int_{D_{R}} \frac{\bar\partial (u(\zeta)-u_j(\zeta))}{\cdot-\zeta}dv_\zeta\right\|_{L^2(D_{R})}\\
        \le &C\|u-u_j\|_{H^1(D_{R})} \\
        =  & C\|u-u_j\|_{H^1(D_{R_0})}\rightarrow 0 
    \end{split}
\end{equation*}
as $j\rightarrow \infty$. Here the validity of the  second inequality was based on the boundedness of the solid Cauchy integral operator 
$$ Tf: = -\frac{1}{\pi }\int_{D_{R}} \frac{f(\zeta)}{z -\zeta }dv_{\zeta}$$
from $L^2(D_R)$ space to itself. This proves \eqref{cg}. 

When $u$ vanishes near  0, the integral \eqref{ll} is well defined. Moreover, $\int_{\mathbb C}\frac{\bar\partial u(\zeta)}{z -\zeta }dv_{\zeta}$ is zero near 0 by \eqref{cg}. Thus
$$  \int_{\mathbb C}\frac{ \bar\partial u(\zeta)}{\zeta^{l+1}}dv_\zeta =-\frac{1}{l!}\left.\frac{\partial^l}{\partial z^l} \int_{\mathbb C} \frac{\bar\partial u(\zeta)}{z-\zeta}dv_\zeta\right|_{z=0} =0. $$

\end{proof}

\begin{proof}[Proof of Theorem \ref{main}: ] 
Without loss of generality, we assume $z_0=0$.  Let $\chi_B$ be the characteristic function for a set $B$.  Fix an $r>0$  small enough such that $D_{2r}\subset \Omega$, $\|u\|_{H^1(D_{2r})}$ is finite, and 
\begin{equation*}
    \left\|V\chi_{D_{r}}+\frac{r}{1+|z|^2}\right\|^2_{L^2(\mathbb C)}\le \frac{\pi^2}{2C_0}, 
\end{equation*} where $C_0$ is the universal constant in Theorem \ref{whl}. Replacing $V$ by $V\chi_{D_r}+\frac{r}{1+|z|^2}$, still denoted by $V$, we have  $\eqref{eqn}$  holds on $D_{2r}$ with  $V\in L^2(\mathbb C)$,  
    \begin{equation}\label{vb}
    V>0 \ \ \text{on}\ \ \mathbb C;   \ \  V\ge C_r \ \  \text{on}\ \ D_{2r}
    \end{equation} 
    for  some positive constant $C_r$ dependent only on $r$, and
    \begin{equation}\label{vs}
    \|V\|^2_{L^2(\mathbb C)}\le \frac{\pi^2}{2C_0}.
    \end{equation}
    We shall show that $u = 0$ on $D_\frac{r}{2}$.  Then the rest of the proof for the strong unique continuation  property  follows from a standard propagation argument. 
    
    Choose $\eta\in C_c^\infty(\mathbb C)$ such that  $\eta =1$ on $D_r$;  $0\le \eta\le 1$ and $|\nabla \eta|\le \frac{2}{r}$ on $D_{2r}\setminus  D_r$; $\eta =0$  outside $D_{2r}$. Let $\psi\in  C^\infty(\mathbb C)$ be  such that $\psi =0$ in $D_1$; $  0\le \psi\le 1$ and $|\nabla \psi|\le 2$ on $D_{2}\setminus {D_1}$; $\psi =1$ outside $D_2$.  For each $k\ge \frac{4}{r}$ (then $ \frac{2}{k}\le\frac{r}{2}$),  let $\psi_k(z) = \psi(kz), z\in\mathbb C$. Note that  $\psi_k \eta u \in H^1(\mathbb C)$ and is supported inside $D_{2r}\setminus D_{\frac{1}{k}}$.  We apply Lemma \ref{prep} to $\psi_k \eta u $ and obtain  for $z\in D_r$,  $m\in \mathbb Z^+,$
\begin{equation*}
    \begin{split}
       |\psi_k(z)|^2 |u(z)|^2 =&\frac{1}{\pi^2} \left |\int_{\mathbb C}\frac{\bar\partial (\psi_k(\zeta)\eta(\zeta)u(\zeta))}{z -\zeta }dv_\zeta \right |^2\\
       =&\frac{1}{\pi^2} \left |\int_{\mathbb C}\left(\frac{1}{z -\zeta } +\sum_{l=0}^{m-1}\frac{z^{l}}{\zeta^{l+1}}\right)\bar\partial \left(\psi_k(\zeta)\eta(\zeta)u(\zeta)\right)dv_{\zeta} \right |^2.
    \end{split}
\end{equation*}
Making use of \eqref{sa}, 
 we  have 
\begin{equation*}
    \begin{split}
        &\int_{D_r}\frac{|\psi_k(z)|^2 |u(z)|^2}{|z|^{2m}}V(z)dv_z\\ =&\frac{1}{\pi^2} \int_{D_r}\frac{1 }{|z|^{2m}}\left |\int_{\mathbb C}\left(\frac{1}{z -\zeta } +\sum_{l=0}^{m-1}\frac{z^{l}}{\zeta^{l+1}}\right)\bar\partial \left(\psi_k(\zeta)\eta(\zeta)u(\zeta)\right)dv_{\zeta} \right |^2V(z)dv_z.\\         \le &\frac{1}{\pi^2} \int_{D_r}\left (\int_{\mathbb C} \frac{1}{|z-\zeta|} \frac{|\bar\partial \left(\psi_k(\zeta)\eta(\zeta)u(\zeta)\right)|}{|\zeta|^m}dv_{\zeta} \right )^2V(z)dv_z\\
        \le & \frac{1}{\pi^2}\left\|I_1\left(\frac{|\bar\partial \left(\psi_k\eta u\right)|}{|\cdot|^{m}} \right)\right\|^2_{L^2_V(\mathbb C)}. 
    \end{split}
\end{equation*}
Applying Theorem \ref{whl} and \eqref{vs}, we further obtain
\begin{equation}\label{mm}
    \begin{split}
       &\int_{D_r}\frac{|\psi_k(z)|^2 |u(z)|^2}{|z|^{2m}}V(z)dv_z\\
        \le &\frac{C_0}{\pi^2}\|V\|^2_{L^2(\mathbb C)}\int_{\mathbb C} \frac{|\bar\partial \left(\psi_k(z)\eta(z)u(z)\right)|^2}{|z|^{2m}} \frac{dv_{z}}{V(z)} \\
        \le &\frac{1}{2}\left(\int_{\mathbb C} \frac{|\bar\partial \psi_k(z)|^2|u(z)|^2}{|z|^{2m}V(z)} dv_{z} + \int_{D_r} \frac{|\psi_k(z)|^2|\bar\partial u(z) |^2}{|z|^{2m}V(z)} dv_{z} +\int_{D_{2r}\setminus D_r} \frac{|\bar\partial \left(\eta(z)u(z)\right) |^2}{|z|^{2m}V(z)} dv_{z}\right)\\
        =&: \frac{A}{2}+\frac{B}{2}+\frac{C}{2}.
    \end{split}
\end{equation}

We first have $\lim_{k\rightarrow \infty} A =0.$ Indeed, since $\bar\partial\psi_k$ is only supported on $D_{\frac{2}{k}}\setminus D_{\frac{1}{k}}$, and   $V>C_r$ on $D_{\frac{2}{k}}(\subset D_r)$,
$$A \le \int_{\frac{1}{k}<|z|<\frac{2}{k}} \frac{|\nabla \psi_k(z)|^2|u(z)|^2}{|z|^{2m}V(z)} dv_{z}\le \frac{k^{2m+2}}{C_r} \int_{|z|<\frac{2}{k}}|u(z)|^2dv_z\rightarrow 0$$
 as $k \rightarrow \infty$, as a consequence of the fact that  $u$ vanishes to infinite order at $0$.

On the other hand, applying the inequality \eqref{eqn} to $B$, we get 
\begin{equation*}
    \frac{B}{2}\le \frac{1}{2}\int_{D_r} \frac{|\psi_k(z)|^2| u(z) |^2}{|z|^{2m}}V(z) dv_{z}.
\end{equation*}
Subtract $\frac{B}{2}$ from both sides of \eqref{mm} and let $k\rightarrow \infty$. Then for each $m\in \mathbb Z^+$,  by Fatou's Lemma, 
 $$\int_{D_r}\frac{ |u(z)|^2}{|z|^{2m}}V(z)dv_z\le\lim_{k\rightarrow \infty}\int_{D_r}\frac{|\psi_k(z)|^2 |u(z)|^2}{|z|^{2m}}V(z)dv_z\le \int_{D_{2r}\setminus D_r} \frac{|\bar\partial \left(\eta(z)u(z)\right) |^2}{|z|^{2m}V(z)} dv_{z}.   $$
 
 Now multiplying both sides of the above inequalities by $r^{2m}$, then
 \begin{equation}\label{nn}
    \int_{D_r}\frac{r^{2m} }{|z|^{2m}}|u(z)|^2V(z)dv_z\le \int_{D_{2r}\setminus D_r} \frac{r^{2m}|\bar\partial \left(\eta(z)u(z)\right) |^2}{|z|^{2m}V(z)} dv_{z}\le \int_{D_{2r}\setminus D_r} \frac{|\bar\partial \left(\eta(z)u(z)\right) |^2}{V(z)} dv_{z}. 
 \end{equation}
 Shrinking the integral domain of the left hand side of \eqref{nn} to $D_{\frac{r}{2}}$, and making use of \eqref{vb}, 
  we further infer
 $$  2^{2m}\int_{D_\frac{r}{2}}|u(z)|^2V(z)dv_z\le \frac{1}{C_r} \int_{D_{2r}}  |\bar\partial \left(\eta(z)u(z)\right) |^2 dv_{z} =\tilde C_r\|u\|^2_{H^1(D_{2r})}, $$
for some  constant $\tilde C_r$ dependent only on $r$. Lastly, letting $m\rightarrow \infty$, we have $u= 0$  on $D_{\frac{r}{2}}$. 

\end{proof}
\medskip

It is worth pointing out that when $V\in L^p, p>2$, the weighted Hardy-Littlewood-Sobolev inequality is not needed to obtain the  unique continuation property for \eqref{eqn}. Indeed,  the strong unique continuation property follows by repeating  the same proof as in Theorem \ref{main}, except with those integrals over the domain $\mathbb C$ substituted by those over $ D_{2r}, r<\frac{1}{2}$, and with Theorem \ref{whl} replaced by the following inequality
\begin{equation}\label{pl}
     \|I_1f\|_{L_V^2(D_1)}\le C_p \|V\|_{L^p(D_1)} \|f\|_{L_{V^{-1}}^2(D_1)} \ \ \text{for all}\ \ f\in L_{V^{-1}}^2(D_1)
\end{equation}
 for a   constant $C_p$ dependent only on $p$.  \eqref{pl} can be  proved by simply using H\"older inequality. This is because when $p>2$, $\|\frac{1}{\cdot -z}\|_{L^q(D_1)}$ is uniformly bounded  on $D_1$ by some constant $C_p$, where  $q$ is the dual of $p$, i.e., $\frac{1}{p}+\frac{1}{q} =1$. Therefore,
 \begin{equation*}
     \begin{split}
     \|I_1f\|^2_{L_V^2(D_1)} \le  & \int_{D_1}\left|\int_{D_1}\frac{|f(\zeta)|}{|\zeta-z|^{\frac{1}{2}}V(\zeta)^{\frac{1}{2}}}\cdot \frac{V(\zeta)^{\frac{1}{2}}}{|\zeta-z|^{\frac{1}{2}}}dv_{\zeta}\right|^2V(z)dv_z\\
     \le &\int_{D_1}\int_{D_1}\frac{|f(\zeta)|^2}{|\zeta-z|V(\zeta)}dv_\zeta\cdot \int_{D_1}\frac{V(\zeta)}{|\zeta-z|}dv_{\zeta}V(z)dv_z\\
     \le &C_p\|V\|_{L^p(D_1)}\int_{D_1}\int_{D_1}\frac{|f(\zeta)|^2}{|\zeta-z|V(\zeta)}dv_\zeta V(z)dv_z\\
     = & C_p\|V\|_{L^p(D_1)}\int_{D_1}\frac{|f(\zeta)|^2}{V(\zeta)}\int_{D_1}\frac{V(z)}{|\zeta-z|}dv_z dv_\zeta\\
     \le & C^2_p\|V\|^2_{L^p(D_1)}\|f\|^2_{L_{V^{-1}}^2(D_1)}.
     \end{split}
 \end{equation*}

Clearly the above approach no longer works when, for instance, the potential $V$ is exactly $L^2$ (namely, in $L^2$ but not in $L^p$ for any $p>2$) as in the following example. For those cases, we have to use Theorem \ref{main} to justify the nonexistence of nontrivial flat  solutions.


\begin{example}\label{ex3}
For  $0<\epsilon <\frac{1}{2}$, let 
$$u_0(z)=e^{-\left(-\ln |z|\right)^\epsilon} \ \ \text{and}\ \  V(z) = 
\frac{\epsilon}{2|z|\left(-\ln |z|\right)^{1-\epsilon}}$$
on  $D_{\frac{1}{2}}$.  Then   $V\in L^2(D_{\frac{1}{2}})$, yet $V\notin L^p(D_{\frac{1}{2}})$  for any $p>2$.  By Theorem \ref{main}, the equation $|\bar\partial u|=V|u|$ has no nontrivial flat solution. Since $u_0$ is continuous on $D_{\frac{1}{2}}$ and satisfies $|\bar\partial u|=V|u|$ in $D_{\frac{1}{2}}\setminus \{0\}$ pointwisely,  by Lemma \ref{ele1} we have $u_0\in H^1(D_\frac{1}{2})$ and   satisfies $|\bar\partial u|=V|u|$  on $D_{\frac{1}{2}}$. Thus $u_0$ can not be flat in particular. On the other hand, one can directly verify that $u_0^{-1}(0) = \{0\}$, yet  $u_0$  fails to vanish to infinite order (in fact, \eqref{flat} fails  for all $m\ge 3$) at $0$.
\end{example}

When $V\in L^p, p<2$,  the strong unique continuation property fails in general as seen below. 

\begin{example}\label{ex}
For each $0<p<2$, choose $\epsilon \in (0, \frac{2-p}{p})$ (so that $(\epsilon+1)p<2$). Then $u_\epsilon(z): = e^{-\frac{1}{|z|^\epsilon}}$ vanishes to infinite order at $0$ and satisfies $ |\bar\partial u| = V|u|$  with
$V = \frac{\epsilon}{2|z|^{\epsilon+1}}\in L^p(D_1)$ on $\mathbb C$. 
\end{example}

More strikingly, Mandache  constructed an example in \cite{Ma02}  with an $L^p, p<2$ potential, where even the weak unique continuation property fails.  

\begin{example}\label{ex4} \cite{Ma02}
  There exist two functions $u\in C_c^\infty(\mathbb C)$ and $ V\in L^p(\mathbb C), 0<p<2,$ such that $u$ is supported on $D_1$ and satisfies $ \bar\partial u =  Vu$.
\end{example}

Despite the above examples in $p<2$,   the first author showed (in Lemma 7,  \cite{Pa92}) that  the strong unique continuation property may still be expected if the target dimension $N=1$ and $V$ takes certain special form, such as a multiple of $\frac{1}{|z|}$. Note that $\frac{1}{|z|}\notin L^2$. 
It should be noted that the strong unique continuation property  no longer holds for this potential  when  $N=2$. In fact,  Wolff and the first author in \cite{PW98} constructed a smooth function $v_0: \mathbb C\rightarrow \mathbb C$ which vanishes to infinite order at $0$ and  satisfies $|\triangle v|\le \frac{C}{|z|}|\nabla v|$ for some constant $C>0$. (See also \cite{AB94} by Alinhac and Baouendi for another example in the same setting.) Letting  $u_0: = (\Re \partial v_0, \Im \partial v_0)$, then $u_0: \mathbb C\rightarrow \mathbb C^2$   vanishes to infinite order at $0$ and  satisfies $|\bar\partial u|\le \frac{C}{|z|}| u|$ for some constant $C$.   
\medskip

\section{Proof of Theorem \ref{main2} and Corollary \ref{main3}}

We shall assume $n\ge 2$ in this section. Note that the inequality \eqref{eqn} reads as \begin{equation}\label{va}
     |\bar\partial u|: =\left(\sum_{j=1}^n\sum_{k=1}^N |\bar\partial_j u_k|^2\right)^{\frac{1}{2}}\le V \left(\sum_{k=1}^N|u_k|^2 \right)^{\frac{1}{2}}: = V|u|. 
\end{equation}
\medskip

\begin{proof}[Proof of Theorem \ref{main2}:] Without loss of generality,  assume  $n=2$. We only need to show that for any  $r_2\ge r_1>0, s>0$, if $u$ satisfies \eqref{eqn} on the bidisk $D_{r_2}\times D_{s}$ and $ u= 0$ on $D_{r_1}\times D_{s}$, then $u =0 $ on $D_{r_2}\times D_{s}$. 

Since $V\in L_{loc}^2(D_{r_2}\times D_{s} )$, by Fubini's theorem, for almost every $z_2\in D_s$, $V(\cdot, z_2)\in L^2_{loc}(D_{r_2})$, and similarly for $u(\cdot, z_2)\in H^1_{loc}(D_{r_2})$. Restricting \eqref{va} at each such $z_2= c_2$, we have  $v: = u(\cdot, c_2)$ on $D_{r_2}$ satisfies
\begin{equation*}
    |\bar\partial v| =\left(\sum_{k=1}^N|\bar\partial_1 u_k(\cdot, c_2)|^2\right)^{\frac{1}{2}}\le |\bar\partial u(\cdot, c_2)| \le V(\cdot, c_2)|u(\cdot, c_2)| =V(\cdot, c_2)|v|.
\end{equation*}
Since  $v= 0 $ on $D_{r_1}$, applying Theorem \ref{main} we have $v= 0 $ on $D_{r_2}$. Thus $u = 0$ on $D_{r_2}\times D_{s}$. The proof is complete. 

\end{proof}

\medskip

\begin{proof}[Proof of Corollary \ref{main3}: ]
Let $u: = f-g$. Then $u\in H^1_{loc}(\Omega)$ vanishes in an open subset and satisfies
$$ \bar\partial u = A u\ \ \text{on}\ \ \Omega.$$
Let $V:  =  \sqrt{\sum_{j,k=1}^N |A_{jk}|^2}$, the matrix norm of $A$. Then  $u$ satisfies $ |\bar\partial u|\le V|u| $ with $V\in L^2_{loc}(\Omega)$. By Theorem \ref{main2}, we have $u\equiv 0$ on $\Omega$. 

\end{proof}

\medskip
One may modify the one-dimensional Example \ref{ex3}  to obtain  an inequality  in higher dimensions with an exact $L^2$ 
potential, where Theorem \ref{main2} can be applied to get the weak continuation property. 

\begin{example}
 For $0<\epsilon <\frac{1}{2}$, let 
$$u_0(z_1, z_2) = e^{-\left(-\ln |z_1|\right)^\epsilon}e^{-\left(-\ln |z_2|\right)^\epsilon}\ \ \text{and}\  \ V(z_1, z_2) = \frac{\epsilon}{2}
\left(\frac{1}{|z_1|^2\left(-\ln |z_1|\right)^{2-2\epsilon}} + \frac{1}{|z_2|^2\left(-\ln |z_2|\right)^{2-2\epsilon}}\right)^\frac12$$
on  $D_{\frac{1}{2}}\times D_{\frac{1}{2}}$. Then $V\in L^2(D_{\frac{1}{2}}\times D_{\frac{1}{2}})$, yet $V\notin L^p(D_{\frac{1}{2}}\times D_{\frac{1}{2}})$  for any $p>2$. Moreover,  since $u_0$ is continuous on $D_{\frac{1}{2}}\times D_{\frac{1}{2}}$ and  satisfies $|\bar\partial u|= V|u|$ in $\left(D_{\frac{1}{2}}\setminus \{0\}\right)\times \left(D_{\frac{1}{2}}\setminus \{0\}\right)$,  $u_0\in H^1(D_\frac{1}{2}\times D_{\frac{1}{2}})$ and   satisfies $|\bar\partial u|= V|u|$ on  $D_{\frac{1}{2}}\times D_{\frac{1}{2}}$ by Corollary \ref{ele2}.  By Theorem \ref{main2}, any nontrivial solution, and thus $u_0$, can not vanish in open subsets of $D_{\frac{1}{2}}\times D_{\frac{1}{2}}$. 
In fact, one can directly verify that $u_0$ only vanishes on $\{z_1=0\}\cup\{z_2 =0\}$ in $D_{\frac{1}{2}}\times D_{\frac{1}{2}}$.  \ \
\end{example}

\medskip

  Similarly, making use of Mandache's example, one can construct an example in $n\ge 2$  where  the  weak unique continuation property fails if the potential is at most $L^p, p<2$.

\begin{example}\label{ex2}
 For each $0<p<2$, there exist nontrivial $u\in C_c^\infty(\mathbb C^2), V\in L^p(\mathbb C^2)$ such that $u$ is supported on $D_1\times D_1$ and satisfies $ |\bar\partial u|\le V|u|$ on $\mathbb C^2$.
\end{example}

\begin{proof}
Let $w, W$ be as in Mandache's Example \ref{ex4} with $w\in C_c^\infty(D_1), W\in L^p(\mathbb C)$ and $ |\bar\partial w|\le W|w|$ on $\mathbb C$. Define $u(z_1, z_2): =w(z_1)w(z_2), (z_1, z_2)\in \mathbb C^2$. Then $u\in C_c^\infty(\mathbb C^2)$ with support in $ D_1\times D_1$. One can check that for $(z_1, z_2)\in \mathbb C^2$, 
\begin{equation*}\begin{split}
    |\bar\partial  u(z_1, z_2)| \le & |\bar\partial_1 w(z_1)||w(z_2)|+ | w(z_1)||\bar\partial_2 w(z_2)|\\
    \le & W(z_1)|w(z_1)||w(z_2)| + W(z_2)|w(z_1)||w(z_2)|\\
    = & \left(W(z_1)\chi_{D_1}(z_2) +W(z_2)\chi_{D_1}(z_2)\right) |u(z_1, z_2)|.
\end{split}
    \end{equation*}
Letting $V(z_1, z_2): = W(z_1)\chi_{D_1}(z_2) +W(z_2)\chi_{D_1}(z_2)$. Then $ |\bar\partial u|\le V|u|$ with
$$\|V\|^p_{L^p(\mathbb C^2)} \le C\left (\int_{\mathbb C}|W(z_1)|^pdv_{z_1} + \int_{\mathbb C}|W(z_2)|^pdv_{z_2}\right)<\infty  $$ for a constant $C>0$.
\end{proof}

The following theorem gives an example where the strong unique continuation property may  hold when the potential $V$ 
is a multiple of $\frac{1}{|z|}$ as below. We note that when $n\ge 2$, $\frac{1}{|z|}$ is  in $L^2$ but not in $ L^{2n}$. Let $B_1$ be the unit ball centered at $0\in \mathbb C^n$.
\medskip

\begin{theorem}
Let $u$ be a smooth function from $B_1\subset \mathbb C^n$ to $\mathbb C$ satisfying
\begin{equation}
    |\bar\partial u|\le \frac{C}{|z|}|u|
\end{equation}
 for some constant $C>0$. If $u$ vanishes to infinite order at $0$, then $u$ vanishes identically on $B_1$. \end{theorem}

\begin{proof}
For each $z\in B_1\setminus\{0\}$, define $w(\tau): =u(\tau z)$, $\tau\in D_{\frac{1}{|z|}}$. Then $w$ vanishes to infinite order at $0$ and satisfies
$$|\bar\partial w(\tau)|\le |z||\bar\partial u(
\tau z)|\le \frac{C}{|\tau|} |w(\tau)|. $$ Making use of Lemma 7 in \cite{Pa92},  we have $w$ vanishes on $D_{\frac{1}{|z|}}$, and thus $u$ vanishes identically on $B_1$.

\end{proof}

\appendix
\section{Appendix}  

The following lemma is known in folklore, and is particularly useful while checking for weak derivatives in examples. For convenience of the reader, we provide a proof below. 
\medskip

\begin{lem}\label{ele1}
Let   $f: D_1\rightarrow \mathbb C$ and $f\in L^1(D_1)$. Suppose $u: D_1\rightarrow \mathbb C$ with $u\in L^2(D_1) $ satisfies 
$$ \bar\partial u  =f \ \ \text{in}\ \ D_1\setminus \{0\}$$
in weak sense. Then $$ \bar\partial u  =f \ \ \text{in}\ \ D_1$$ in weak sense. \end{lem}

\begin{proof}
 Let $\eta \in C_c^\infty(D_1)$ be a cut-off function with $\eta=1$ on $D_\frac{1}{2}$ and $|\nabla \eta|<3$ on $D_1$. For each $k>0$,  let $\eta_k(z) =  \eta(kz), z\in\mathbb C$.  Given a fixed testing function $\phi\in C_c^\infty(D_1)$,
 \begin{equation*}
     \begin{split}
         \int_{D_1}u(z)\bar\partial  \phi(z)dv_z = \int_{D_1} u(z)\bar\partial \left(\eta_k(z)\phi(z)\right)dv_z +\int_{D_1}  u(z)\bar\partial \left(\left(1-\eta_k(z)\right)\phi(z)\right)dv_z: = A+B.
     \end{split}
 \end{equation*}
Since $(1-\eta_k)\phi$ is a testing function on $D_1\setminus \{0\}$, by assumption
$$B = -\int_{D_1} f(z)\left(1-\eta_k(z)\right) \phi(z)dv_z \rightarrow -\int_{D_1} f(z) \phi(z)dv_z$$
when $k\rightarrow \infty$. For $A$, since $\eta_k$ and  $|\bar\partial \eta_k|$ are only supported on $D_{\frac{1}{k}}$ with $|\eta_k|\le 1$ and  $|\bar\partial \eta_k|\le 3k$, we have by H\"older inequality
\begin{equation*}
    \begin{split}
        |A|\le &\int_{D_1} |u(z)||\bar\partial \eta_k(z)||\phi(z)|dv_z + \int_{D_1} |u(z)||\eta_k(z)||\bar\partial \phi(z)|dv_z\\
        \le & 3k\int_{D_{\frac{1}{k}}} |u(z)\phi(z)|dv_z + \int_{D_{\frac{1}{k}}} |u(z)\bar\partial \phi(z)|dv_z\\
        \le &3k\frac{\sqrt{\pi}}{k}\|u\phi\|_{L^2(D_{\frac{1}{k}})} + \frac{\sqrt{\pi}}{k}\|u\bar\partial\phi\|_{L^2(D_{\frac{1}{k}})}\\
        \le& C\|u\|_{L^2(D_{\frac{1}{k}})}
    \end{split}
\end{equation*}
 for some constant $C>0$ (dependent only on $\phi$). Letting $k\rightarrow \infty$, we get $|A|\rightarrow 0$. The proof is complete.
 
\end{proof}
\medskip

We remark that  the assumption $u \in L^2(D_1)$ in Lemma \ref{ele1} can not be further relaxed, as indicated by the following example. 

\begin{example}
Let $u_0(z) = \frac{1}{z}, z\in D_1$. Then $u_0$ satisfies $$ \bar\partial u  = 0\ \ \text{in}\ \ D_1\setminus \{0\}.$$
However, $u_0$ is not a weak solution to $
    \bar\partial u  = 0\ \ \text{in}\ \ D_1.$
\end{example}

\begin{proof}
We only need to verify for a general $\phi\in C_c^\infty(D_1)$, $\int_{D_1}\frac{\bar\partial \phi(z)}{z} dv_z \ne 0. $ Indeed, by Stokes' theorem, 
\begin{equation*}\begin{split}
     \int_{D_1}\frac{\bar\partial \phi(z)}{z} dv_z = &\lim_{\epsilon\rightarrow 0}  \int_{D_1\setminus D_\epsilon}\frac{\bar\partial \phi(z)}{z} dv_z \\
    = &\frac{1}{2}\lim_{\epsilon\rightarrow 0}  \int_{D_1\setminus D_\epsilon}\bar\partial\left( \frac{ \phi(z)}{z}dz\right) \\
     =& -\frac{1}{2}\lim_{\epsilon\rightarrow 0}  \int_{\partial D_\epsilon} \frac{ \phi(z)}{z} dz\\
     =& -\phi(0)\pi i.
     \end{split}
   \end{equation*}
\end{proof}

One can immediately extend Lemma \ref{ele1} to higher dimensions as follows. 
\medskip

\begin{cor}\label{ele2}
Let   $f_j : D_1\rightarrow \mathbb C$ and $f_j\in L^1(D_1)$, $j=1, 2$. Suppose  $u_j: D_1\rightarrow \mathbb C$ with $u_j\in L^2(D_1) $ and satisfies 
$$ \bar\partial u_j  =f_j \ \ \text{in}\ \ D_1\setminus \{0\}$$
in weak sense for each $j$. Then $u(z_1, z_2): =u_1(z_1)u_2(z_2)$ satisfies $$ \bar\partial_1 u(z_1, z_2)  =f_1(z_1)u_2(z_2) \ \  \text{and} \ \ \bar\partial_2 u (z_1, z_2) =u_1(z_1)f_2(z_2) \ \ \text{in}\ \ D_1\times D_1$$ in weak sense.
\end{cor}

\begin{proof}
We  only prove  the first equation  holds in $D_1\times D_1$ in weak sense. For a fixed testing  function $\phi\in C_c^\infty(D_1\times D_1)$, by Fubini's theorem, 
\begin{equation*}
     \int_{D_1\times D_1}u(z_1, z_2)\bar\partial_1 ( \phi(z_1, z_2))dv_z = \int_{D_1}u_2(z_2)\int_{D_1} u_1(z_1)\bar\partial_1 \left(\phi(z_1, z_2)\right)dv_{z_1}dv_{z_2}. 
      \end{equation*}
Since $\phi(\cdot, z_2)\in C_c^\infty(D_1)$  for each $z_2\in D_1$, applying Lemma \ref{ele1}  to $u_1$, we get 
\begin{equation*}
    \begin{split}
       \int_{D_1\times D_1}u(z_1, z_2)\bar\partial_1 ( \phi(z_1, z_2))dv_z =& - \int_{D_1}u_2(z_2)\int_{D_1} f_1(z_1)\phi(z_1, z_2)dv_{z_1}dv_{z_2}\\
       =&  -\int_{D_1\times D_1} f_1(z_1)u_2(z_2) \phi(z_1, z_2)dv_z.
    \end{split}
\end{equation*}

\end{proof}

\bibliographystyle{alphaspecial}

\fontsize{11}{11}\selectfont

\noindent pan@pfw.edu,

\vspace{0.2 cm}

\noindent Department of Mathematical Sciences, Purdue University Fort Wayne, Fort Wayne, IN 46805-1499, USA.\\
\vspace{0.2cm}

\noindent zhangyu@pfw.edu,

\vspace{0.2 cm}

\noindent Department of Mathematical Sciences, Purdue University Fort Wayne, Fort Wayne, IN 46805-1499, USA.\\
\end{document}